\documentclass[11pt]{amsart}
\usepackage{amssymb}
\usepackage{mathrsfs}
\usepackage{amsmath}
\usepackage{amsfonts}
\usepackage{amsfonts,amssymb}
\usepackage{amscd}
\usepackage{amsthm,amsmath}
\usepackage[dvips]{graphicx}
\usepackage{xy}
\usepackage{latexsym,amssymb}  
\usepackage{graphicx}
\usepackage{kpfonts} 

\def\otm{\otimes}

\def\Z{\mathbb Z}

\def\lr#1{[\,#1\,]}
\def\lra#1{\langle\,#1\rangle}
\def\w{\mathfrak{w}}

\newtheorem{thm}{Theorem}[section]
\newtheorem{lem}[thm]{Lemma}
\newtheorem{prop}[thm]{Proposition}
\newtheorem{cor}[thm]{Corollary}
\newtheorem{defn}[thm]{Definition}
\newtheorem{exam}[thm]{Example}

\numberwithin{equation}{section}

\date{}
\begin{document}

\thispagestyle{empty}

\begin{center}
{\bf{ \LARGE  On $4n$-dimensional neither pointed nor semisimple Hopf algebras and the associated weak Hopf algebras} \footnotetext { $\dag$
Corresponding author: slyang@bjut.edu.cn}}


\bigbreak

\normalsize Jialei Chen$^a$, Shilin Yang$^{a\dag}$, Dingguo Wang$^b$, Yongjun Xu$^b$

{\footnotesize\small\sl $^a$School of Mathematics, Statistics and Mechanics,
Beijing University of Technology, Beijing 100124, P. R. China\\ }

{\footnotesize\small\sl $^b$School of Mathematical Sciences,
Qufu Normal University, Qufu 273165, P. R. China\\ }
\end{center}
\begin{quote}
{\noindent\small{\bf Abstract.}
For a class of neither pointed nor semisimple Hopf algebras $H_{4n}$ of dimension $4n$, it is shown that they are quasi-triangular, which universal $R$-matrices are described. The corresponding weak Hopf algebras $\w H_{4n}$ and their representations are constructed. Finally, their duality and their Green rings are established by generators and relations explicitly. It turns out that the Green rings of the associated weak Hopf algebras
 are not commutative even if the Green rings of $H_{4n}$ are commutative.

}
\end{quote}

\section*{Introduction}
Let $k$ be an algebraically closed field of characteristic zero.
In this paper, we will study representations of the class of neither pointed nor semisimple Hopf algebra $H_{4n}$ of dimension $4n$  (see Definition \ref{H4n})  and the associated weak Hopf algebras.

   The class of Hopf algebras $H_{4n}$ plays an important role in constructing new Nichols
algebras, new Hopf algebras and classifying Hopf algebras. Note that if $n=2, 3$ and $a\neq 0$, it is just
the unique neither pointed nor-semisimple $8$-dimensional Hopf algebra $(A_{C_4}^{\prime\prime})^\ast$
 (see \cite{Wa}), or the $12$-dimensional Hopf algebra $A_{1}^\ast$ (see \cite{NAT}) up to isomorphism respectively.
 In \cite{GG}, the authors determined all finite-dimensional Hopf algebras over $k$ whose coradical generated a Hopf subalgebra isomorphic to $H_8$. They also obtained new Nichols algebras of dimension $8$ and new Hopf algebras of dimension $64$. Based on this, \cite{RX} determined all finite-dimensional Nichols algebras over the semisimple objects in $^{H_8}_{H_8}YD$ and obtained some new Nichols algebras of non-diagonal type and new Hopf algebras without the dual Chevalley property. By the equivalence $_{M}D(H_{12})\simeq ^{H_{12}}_{H_{12}}YD$, the authors (\cite{HX16},\cite{RX17}) obtained some new Nichols algebras which were not of diagonal type and some families of new
Hopf algebras of dimension $216$.

As is well known, the classification of finite dimensional Hopf algebras over $k$ is an important open problem. Since Kaplansky's conjectures posed in 1975, several results on them have been obtained (see \cite{Zhu, DS, AN, NG02, NG05, NG09, CNG, BG}). In \cite{AN}, the authors proved that there were exactly $4(q-1)$ isomorphism classes of non-semisimple pointed Hopf algebras of dimension
$pq^2$, of which, those Radford's Hopf algebras (see \cite{RAD}) occupied $1/4$.
 It is remarked that the dual of $H_{4n}$  is just the Radford's Hopf algebra in \cite{RAD}.

Given a Hopf algebra $H$, the
decomposition problem of tensor products of indecomposable modules has attracted numerous attentions. In \cite{CIBILS}, Cibils classified the indecomposable modules over $k\mathbb{Z}_{n}(q)/I_{d}$, and gave the decomposition formulas of the tensor product of two indecomposable $k\mathbb{Z}_{n}(q)/I_{d}$-modules. Yang determined the
representation type of a class of pointed Hopf algebras, classified all indecomposable modules of simple pointed Hopf algebra $R(q,a)$.
The decomposition formulas of the tensor product of two indecomposable $R(q,a)$-modules is given(see \cite{YANG2}) . It is noted that some results of $R(q,a)$ were recently extended to more general case of pointed Hopf algebras of rank one by Wang et al. (see \cite{WLZ}). Li and Hu described the Green rings of the 2-rank Taft algebra(at $q=-1$) (see \cite{LH}). Chen, Van Oystaeyen and Zhang gave the Green rings of the Taft algebra $H_{n}(q)$ (see \cite{COZ}). Li and Zhang extended the results of \cite{COZ}, computed the Green rings of the Generalized Taft Hopf algebras $H_{n,d}$ by generators and generating relations, and determined all nilpotent elements in $r(H_{n,d})$ (see \cite{LZ}). Su and Yang (see \cite{SY2}) characterized the representation ring of small quantum group $\bar{U}_q{(sl_2)}$
by generators and relations. It turns out that the representation ring of $\bar{U}_q{(sl_2)}$ is generated by infinitely many generators subject to a family of generating relations.

The concept of weak Hopf algebra in the sense of Li was introduced by \cite{LF} in 1998 as a generalization of Hopf algebra. Since then, many weak Hopf algebras or weak quantum groups were constructed, for example, Aizawa and Isaac (\cite{AIS}) constructed weak Hopf algebras corresponding to $U_q (sl_n)$ and Yang (\cite{YANG1}) constructed weak Hopf algebras $\w^d_q (\mathfrak{g})$ corresponding to quantized enveloping algebras $U_q (\mathfrak{g})$ of a finite dimensional semisimple Lie algebra $\mathfrak{g}$. In \cite{SY}, Su and Yang constructed the weak Hopf algebra $\widetilde{H_8}$ corresponding to the non-commutative and non-cocommutative semisimple Hopf algebra $H_8$ of dimension 8. They described  the representation ring of $\widetilde{H_8}$ and studied the automorphism group of $r(\widetilde{H_8})$. In \cite{SY1}, Su and Yang
studied the Green ring of the weak Generalized Taft Hopf algebra $r(\mathfrak{w}^{s}(H_{n,d}))$, showing that the Green ring of the weak Generalized Taft Hopf algebra was much more complicated than its Grothendick ring.

In the present paper, it is shown that $H_{4n}$ is quasi-triangular, which universal $R$-matrices are described.
the weak Hopf algebras $\w H_{4n}$ and $\w H_{4n}^*$ corresponding to the Hopf algebra $H_{4n}$ and its dual $H_{4n}^*$ are constructed.  Then their representations and Green rings are explicitly described.
It turns out that the Green rings of the associated weak Hopf algebras are not commutative even if the
Green rings of $H_{4n}$ are commutative.

The paper is organized as follows. In Section 1, the definition of $H_{4n}$ by generators and relations
is described first, then we prove that $H_{4n}$ is quasitriangular and describe all universal $R$-matrices $R$ explicitly. In Section 2, we compute the Green ring $r(H_{4n})$. In Section 3, we construct the weak Hopf algebra $\w H_{4n}$ associated to $H_{4n}$. In Section 4, we study the representation ring $r(\w H_{4n})$ of $\w H_{4n}$ by generators and relations explicitly. In Section 5, we consider the dual Hopf algebra $H_{4n}^*$
and its weak Hopf algebra $\w H_{4n}^*$, we also describe the representation rings
$r(H_{4n}^*)$ and $r(\w H_{4n}^*)$.


Throughout this paper, we work over an algebraically closed field $k$ of characteristic zero. For the theory of
Hopf algebras and quantum groups, we refer to \cite{MONT,SW,KA,MA}.

\section{The non-semisimple non-pointed Hopf algebra $H_{4n}$}

First of all, let us give the defintion of the Hopf algebra $H_{4n}$.	
\begin{defn}\label{H4n}
Let $n\ge 1$ and $q$ be a primitive $2n$-th root of unity.	
The Hopf algebra $H_{4n}$ is  defined as follows.
As an algebra it generated by
$z, x$ with relations
$$z^{2n}=1, \quad zx=qxz,\quad  x^2=0$$
for any $a\in k$.

The coalgebra structure is
\begin{eqnarray*}
  &&\Delta(z)=z\otm z+a(1-q^{-2})z^{n+1}x\otm zx, \quad \Delta(x)=x\otm 1+z^n\otm x; \\
  &&\epsilon(z)=1, \quad \epsilon(x)=0,\\
 &&S(z)=z^{-1}, \quad  S(x)=-z^nx.
\end{eqnarray*}
\end{defn}	
One sees that $H_4$ is just the $4$-dimensional Sweedler$'$s Hopf algebra  when
$n=1$. It is well known
that $H_4$ is quasi-triangular with universal $R$-matrix
$$R=\frac{1}{2}\left(1\otm 1+1\otm z+z\otm 1+z\otm z\right)
+\alpha\left(x\otm x-x\otm zg+zx\otm x+zx\otm zx\right)$$
for any $\alpha\in k$.

In the sequel, we always assume that $n>1$.  Therefore $q^2\ne 1$ and
$$\Delta(z^i)=z^i\otm z^i+a(1-q^{-2})(1+q^{-2}+\cdots +q^{-2(i-1)})z^{n+i}x\otm z^ix.$$
Let $C_i$ be the $k$-space spanned by $z^i, \ z^{n+i}x, \ z^ix, \ z^{n+i} (1\leq i\leq n-1)$  and
$$T_4=k1\oplus kz^n\oplus kz^nx\oplus kx.$$
\begin{lem} If $a\ne 0$, then
$C_i$ is a simple subcoalgebra and as coalgebras
$$H_{4n}=\bigoplus_{i=1}^{n-1}C_i\oplus T_4$$
and $T_4\cong H_4$ as Hopf algebras.
\end{lem}
\begin{proof} It is  straightforward.
\end{proof}

It follows that if $a\ne 0$, the Hopf algebra $H_{4n} (n\ge 2)$ is not pointed.
\begin{exam}\label{8-dim}
If $q$ is $4$-th primitive root of unity and $a=2$, then $H_8$ is just
the unique neither pointed nor semisimple $8$-dimensional Hopf algebra $(A_{C_4}^{\prime\prime})^\ast$
 {\rm(see \cite{Wa})} up to isomorphism.
\end{exam}

\begin{exam}\label{12-dim}
If $q$ is $6$-th primitive root of unity and $a\ne 0$, then $H_{12}$ is just
the  unique neither pointed nor semisimple $12$-dimensional Hopf algebra $A_{1}^\ast$ \rm{(see \cite{NAT})}.
\end{exam}

By \cite[Lemma 2.2, Theorem 2.1]{YANG2}, $H_{4n}$ is a Nakayama algebra with $2n$ cyclic orientation and cyclic relations of length $2$. In particular, it is of finite representation type.

For every integer $j$, we set $$E_j = \frac{1}{2n} \sum_{i=0}^{2n-1} q^{-ij} z^i.$$
	
It is easy to see that $E_0, \hdots, E_{2n-1}$ list the distinct $E_i's$. Moreover, for  $0\leq j,k < 2n$, we have
	\begin{equation}\label{multx} E_jz^ k  = \frac{1}{n}\sum_{i=0}^{2n-1} q^{-ij} z^{i+k} =
	q^{jk} \left(\frac{1}{n}\sum_{i=0}^{2n-1} q^{-(i+k)j} z^{i+k}\right) =q^{jk}E_j.
\end{equation}
and
$$xE_i=E_{i+1}x.$$
	
	\begin{lem}\label{completeset}
		$\{E_0, \cdots, E_{2n-1}\}$ is a complete set of orthogonal idempotents of $H_{4n}$.
	\end{lem}
	\begin{proof}
		Since $q^{-j}$ is also an $2n$-th root of unity different from $1$ if $j\neq 0$, we get
		$$ \sum_{i=0}^{n-1} E_i = \frac{1}{2n} \sum_{i=0}^{2n-1}\sum_{j=0}^{2n-1} q^{-ij}z^j
		= \frac{1}{2n}\sum_{j=0}^{2n-1}  \left( \sum_{i=0}^{n-1}\left(q^{-j}\right)^i \right) z^j = 1,$$
		
		Also, using (\ref{multx}), for $0\leq l, j <2n$:
		$$ E_jE_l = \frac{1}{2n}\sum_{k=0}^{2n-1} q^{-lk} E_jz^k = \frac{1}{2n}\sum_{k=0}^{2n-1} q^{-lk+jk}E_j = \frac{1}{2n} \sum_{k=0}^{2n-1} \left(q^{j-l}\right)^k E_j = \left\{\begin{array}{cc} E_j & \mbox{if } l=j\\ 0 & \mbox{if } l\neq j\end{array}\right.$$
		
		Hence, $\{E_0, \cdots, E_{2n-1}\}$ is a complete set of orthogonal idempotents of $H_{4n}$.
	\end{proof}

Quasi-triangular Hopf algebras play an important in the theory of Hopf algebras and quantum groups, since they provide solutions to quantum Yang-Baxter equations. People try to construct quasi-triangular Hopf algebras and get a lot of results(see \cite{Wa,NA,CW,WL,LW2}). In this section, we shall show that $H_{4n}$ is quasitriangular and give all universal $R$-matrices explicitly. First, we recall the definition of quasi-triangular Hopf algebra.

Let $H$ be a finite dimensional Hopf algebra and $R\in H\otimes H$ an invertible element. The pair
$(H, R)$ is said to be a quasi-triangular Hopf algebra and $R$ is said to be a universal $R$-matrix of $H$, if the following three conditions are satisfied.
\begin{itemize}
  \item[(i)] $\Delta^{\prime}(h)=R\Delta(h) R^{-1},$ for all $h\in H$;
  \item[(ii)] $(\Delta\otimes id) (R)=R_{13}R_{23}$;
  \item[(iii)] $(id \otimes \Delta ) (R)=R_{13}R_{12}$;
\end{itemize}
Here $\Delta^{\prime}=T\circ\Delta, T: H\otimes H\to H\otimes H, T(a\otimes b)=b\otimes a$, and $R_{ij}\in H\otimes H\otimes H$ is given
by $R_{12}=R\otimes 1$, $R_{23}=1\otimes R$,  $R_{13}=(T\otimes id)(R_{23})$.

\begin{thm} $H_{4n}$ is a quasi-triangular Hopf algebra with universal $R$-matrix
$$R=\sum_{i, j=0}^{2n-1} (-1)^{ij}E_i\otm E_j+2a\sum_{i, j=0}^{2n-1}(-1)^{i(j+1)}E_i x\otm E_j x.$$
\end{thm}

\begin{proof}
Note that here $n>1$ and  $q^2\ne 1$.

Let $R\in H_{4n}\otimes H_{4n}$ be a universal $R$-matrix, and $T=k\lra{z|z^{2n}=1}$.
First of all, we claim that
$$R\in T\otm T+(T\otm T)(x\otm x).$$
Indeed, we assume that
$$R=\sum_{h\in T} h\otm X_h+\sum_{h\in T} hx\otm Y_h, X_h, Y_h\in H_{4n}.$$
Note that
$\Delta(z^n)=z^n\otm z^n$
and $\Delta^{\rm cop}(z^n)R=R\Delta(z^n), $
The relations $zx=qxz$ implies that $xz^n=-z^n x.$
From this relation, it follows that $X_h\in T$ and $Y_h\in Tx.$ Hence
$R$ can be written as
$R=R'+\hat{R}$ where $R'\in T\otm T$ and $\hat{R}\in (T\otm T)(x\otm x).$
Let
 $$R'=\sum_{i, j=0}^{2n-1} a_{ij} E_i\otm E_j\in T\otm T.$$
Note that $(\epsilon\otm id)(\hat{R})=0$ and
 $(\epsilon\otm id)(R)=1$, therefore $(\epsilon\otm id)(R')=1$. Thus we have
 $$a_{i0}=a_{0j}=1  \hbox {\quad for all  \quad }  i, j=0, 1, \cdots, 2n-1.$$
Moreover, since $\Delta^{cop}(x)R=R\Delta(x)$, and $\Delta^{cop}(x)\hat{R}=0=\hat{R}\Delta(x)$,
we see that $\Delta^{cop}(x)R'=R'\Delta(x)$,
 \begin{eqnarray*}
   \sum_{i,j=0}^{2n-1} a_{ij} E_i\otm xE_j+\sum_{i,j=0}^{2n-1} a_{ij} xE_i\otm z^nE_j
    =\sum_{i,j=0}^{2n-1} a_{ij} E_ix\otm E_j+\sum_{i,j=0}^{2n-1} a_{ij} E_iz^n\otm E_jx
 \end{eqnarray*}
 Hence we get
 \begin{eqnarray*}
   \sum_{i,j=0}^{2n-1} a_{ij} E_i\otm E_{j+1}x+\sum_{i,j=0}^{2n-1} (-1)^j a_{ij} E_{i+1}x\otm E_j
    =\sum_{i,j=0}^{2n-1} (-1)^ia_{ij} E_i\otm E_jx+\sum_{i,j=0}^{2n-1} a_{ij} E_ix\otm E_j.
 \end{eqnarray*}
 This implies that
 $$a_{i,j-1}=(-1)^i a_{ij}, \mbox{ and }  a_{i-1,j}=(-1)^j a_{ij}$$
and we have $a_{ij}=(-1)^{ij}.$
 Then
 any universal $R$-matrix $R$ of $H_{4n}$ can be expressed by
 $$R=\sum_{i, j=0}^{2n-1} (-1)^{ij} E_i\otm E_j+ \hat{R},$$
 where $\hat{R}$ can be written as
 $$\hat{R}=\sum_{i, j=0}^{2n-1}b_{ij} E_ix\otm E_j x,  \quad b_{ij}\in k.$$
It is noted that
$$\Delta(z)=(1\otm 1+a(q^2-1)z^nx\otm x)(z\otm z).$$
Compute both side of the equation
$$\Delta^{cop}(z)R=R\Delta(z),$$
then it is straightforward to see that
the left hand side is
$$\sum_{i, j=0}^{2n-1} (-1)^{ij} q^{i+j} E_i\otm E_j+
\sum_{i, j=0}^{2n-1} \left[a(q^2-1)(-1)^{(i-1)(j-1)+j}q^{i+j-2}+b_{ij}q^{i+j}\right] E_ix\otm E_jx,$$
and the right hand side is
$$\sum_{i, j=0}^{2n-1} (-1)^{ij} q^{i+j} E_i\otm E_j+
\sum_{i, j=0}^{2n-1} \left[a(q^2-1)(-1)^{i+ij}q^{i+j-2}+b_{ij}q^{i+j-2}\right] E_ix\otm E_jx.$$
Comparing the two-hand side of the above equation, we have
$$a(q^2-1)(-1)^{(i-1)(j-1)+j}q^{i+j-2}+b_{ij}q^{i+j}=a(q^2-1)(-1)^{i+ij}q^{i+j-2}+b_{ij}q^{i+j-2},$$
and we get
$$b_{ij}=2a(-1)^{ij+i}.$$
Hence, if $R$ is a universal $R$-matrix of $H_{4n}$, then $R$ must be equal to
$$R=\sum_{i, j=0}^{2n-1} (-1)^{ij}E_i\otm E_j+2a\sum_{i, j=0}^{2n-1}(-1)^{i(j+1)}E_i x\otm E_j x.$$
By direct computations we see that $\left(\Delta\otm id\right)(R)=R_{13}R_{23}$ and $\left(id\otm\Delta\right)(R)=R_{13}R_{12}$. Hence  $R$ is a universal $R$-matrix of $H_{4n}$.
\end{proof}

\section{Indecomposable representations of $H_{4n}$}
From this section, we always assume that $a\neq 0$ in Definition \ref{H4n}. The situation for $a=0$ can be considered similarly. Let $H=H_{4n}$ and $M_i$ be the $2$-dimensional cyclic $H$-module with bases $\{v_{1}^i, v_{2}^i\}$, where $i\in\Z_{2n}$. The multiplication of $x$ and
$z$ in $H$ provides the actions on $M_i$, that is
 \begin{eqnarray*}
x(v_1^{i},v_2^{i})&=&(v_1^{i},v_2^{i})\left(
                                          \begin{array}{cc}
                                            0 & 0 \\
                                            1 & 0 \\
                                          \end{array}
                                        \right),\\
z(v_1^{i},v_2^{i})&=&(v_1^{i},v_2^{i})\left(
                                          \begin{array}{cc}
                                            q^{i} &0 \\
                                            0 & q^{i+1} \\
                                          \end{array}
                                        \right).
\end{eqnarray*}
 For any $i\in\Z_{2n}$, let $S_i$ be the $1$-dimensional cyclic $H$-module with base $\{v_i\}$, with the action $x\cdot v_i=0, z\cdot v_i=q^iv_i.$ Up to isomorphism, $\{M_i| i\in\Z_{2n}\}$ provides the complete list of isomorphism classes of indecomposable $H$-modules with two dimension.
Then we have the following decomposition formulas of
the tensor product of two indecomposable $H$-modules.
\begin{thm}\label{thm3-1} Let $i,j\in \Z_{2n}$, then as $H$-modules, we have
\begin{enumerate}
  \item
  $S_i \otimes S_j\cong S_{i+j(\mathrm{mod}2n)}.$
\item
  $S_i \otimes M_j\cong M_{i+j(\mathrm{mod}2n)}.$
   \item
    $M_i \otimes M_j\cong M_{i+j(\mathrm{mod}2n)}\oplus M_{i+j+1(\mathrm{mod}2n)}.$
\end{enumerate}
\end{thm}
\begin{proof}
Recall that
$\Delta(z)=z\otm z+a(1-q^{-2})z^{n+1}x\otm zx,$ and $\Delta(x)=x\otm 1+z^n\otm x$, for $i\in\Z_{2n}$, let $\sigma(i)=(-1)^i$, we have

(1) $x\cdot (v_i\otimes v_j)=0$, $z\cdot (v_i\otimes v_j)=q^{i+j}v_i\otimes v_j$,
therefore $S_i \otimes S_j\cong S_{i+j(\mathrm{mod}2n)}.$

(2) For $j,k\in\{1,2\}$ and $i\in\Z_{2n}$,
\begin{eqnarray*}
&&x\cdot (v_i\otimes v_{k}^j)=\left\{
\begin{array}{ll}
\sigma(i)v_i\otimes v_{2}^j, & k=1,   \\
0, & k=2.
\end{array}
\right.\\
&&z\cdot (v_i\otimes v_{k}^j)=\left\{
\begin{array}{ll}
q^{i+j}v_i\otimes v_{k}^j, & k=1,   \\
q^{i+j+1}v_i\otimes v_{k}^j, & k=2.
\end{array}
\right.
\end{eqnarray*}
so we have  $S_i \otimes M_j\cong M_{i+j(\mathrm{mod}2n)}.$

(3) For $k,l\in\{1,2\}$ and $i\in\Z_{2n}$, note that
 \begin{eqnarray*}
&&x\cdot (v_{1}^i\otimes v_{1}^j)=v_{2}^i\otimes v_{1}^j+\sigma (i)v_{1}^i\otimes v_{2}^j,\\
&&x\cdot (v_{1}^i\otimes v_{2}^j)=v_{2}^i\otimes v_{2}^j,\\
&&x\cdot (v_{2}^i\otimes v_{1}^j) =\sigma (i+1)v_{2}^i\otimes v_{2}^j,\\
&&x\cdot (v_{2}^i\otimes v_{2}^j) =0,\\
&&z\cdot (v_{k}^i\otimes v_{l}^j)=\left\{
                                    \begin{array}{ll}
                                      q^{i+j}\big(v_{1}^i\otimes v_{1}^j+a(q^2-1)\sigma(i+1) v_{2}^i\otimes v_{2}^j\big), & k+l=2 \\
                                      q^{i+j+1} v_{k}^i\otimes v_{l}^j, & k+l=3;\\
                                      q^{i+j+2} v_{2}^i\otimes v_{2}^j, & k+l=4.
                                    \end{array}
                                  \right.
\end{eqnarray*}
Let $w_{1}=v_{1}^i\otimes v_{2}^j$, $w_{2}=v_{2}^i\otimes v_{2}^j$, and
 \begin{eqnarray*}
w_{3}&=&v_{1}^i\otimes v_{1}^j-a\sigma(i+1)v_{2}^i\otimes v_{2}^j, \\
w_{4}&=&v_{2}^i\otimes v_{1}^j+\sigma(i)v_{1}^i\otimes v_{2}^j,
\end{eqnarray*}
then we have
  \begin{eqnarray*}
 x(w_{1},w_2)&=&(w_1,w_2)\left(
                                          \begin{array}{cc}
                                            0 & 0 \\
                                            1 & 0 \\
                                          \end{array}
                                        \right),\\
z(w_1,w_2)&=&(w_1,w_2)\left(
                                          \begin{array}{cc}
                                            q^{i+j+1} &0 \\
                                            0 & q^{i+j+2} \\
                                          \end{array}
                                        \right),
\end{eqnarray*} and
 \begin{eqnarray*}
 x(w_{3},w_4)&=&(w_3,w_4)\left(
                                          \begin{array}{cc}
                                            0 & 0 \\
                                            1 & 0 \\
                                          \end{array}
                                        \right),\\
z(w_3,w_4)&=&(w_3,w_4)\left(
                                          \begin{array}{cc}
                                            q^{i+j} &0 \\
                                            0 & q^{i+j+1} \\
                                          \end{array}
                                        \right).
\end{eqnarray*}
  Therefore, $M_i \otimes M_j\cong M_{i+j(\mathrm{mod}2n)}\oplus M_{i+j+1(\mathrm{mod}2n)}.$
\end{proof}
Let $H$ be a finite dimensional Hopf algebra and $M$ and $N$ be two finite dimensional $H$-modules.
Recall that the Green ring or the representation ring $r(H)$ of $H$ can be defined as follows. As a group $r(H)$ is the free Abelian group generated by the isomorphism classes of the finite dimensional $H$-modules $M$, modulo the relations $[M \oplus N ] = [M] + [N ]$. The multiplication of $r(H)$ is given by the tensor product of $H$-modules, that is, $[M][N]=[M\otimes N]$. Then $r(H)$ is an associative ring with identity given by
$[k_\varepsilon]$, the trivial 1-dimensional $H$-module. Note that $r(H)$ is a
free abelian group with a $\mathbb{Z}$-basis $\{[M ]|M \in ind(H)\}$, where $ind(H)$ denotes
the set of finite dimensional indecomposable $H$-modules.

Denote $[S_1]=b$, $[M_0]=c$.
\begin{cor} \label{cor3-2}
The Green ring $r(H_{4n})$ is a commutative ring generated by $b$ and $c$. The set $\{b^{k}\mid 0 \leq k \leq 2n-1\}\cup \{
b^{i}c \mid 1 \leq i\leq 2n-1\}$ forms a $\mathbb{Z}$-basis for $r(H_{4n})$.
\end{cor}
\begin{proof}
Firstly, $r(H_{4n})$ is a commutative ring since $H_{4n}$ is a quasitriangular Hopf algebra. By Theorem \ref{thm3-1}, $b^{2n}=1$ and there is a one to one correspondence between the set $\{b^{i}\mid 0 \leq i \leq 2n-1\}$ and the set of one-dimensional simple $H_{4n}$ module $\{[S_{i}]\mid 0 \leq i \leq 2n-1\}$. Besides, for all $0 \leq i \leq 2n-1$, $[S_{i}]c=[M_{i}]$, hence $[M_{i}]=b^ic$ and all the two-dimensional simple $H_{4n}$ modules $\{[M_{i}]\mid 0 \leq i <j\leq n-1\}$ are obtained.
\end{proof}

\begin{thm}\label{thm3-2} The Green ring $r(H_{4n})$ is isomorphic to the quotient ring of the ring $\mathbb{Z}[x_1, x_2]$ module the ideal $I$ generated by the following elements
$$x_1^{2n}-1,\quad x_2^2-x_1x_2-x_2$$
\end{thm}
\begin{proof}
By Corollary \ref{cor3-2}, $r(H_{4n})$ is generated by $b $ and $c$. Hence there is a unique ring epimorphism
$$\Phi: \mathbb{Z}\lr{x_1,x_2}\rightarrow r(H_{4n})$$
such that
$$\Phi(x_1)=b=[S_{1}],\quad \Phi(x_2)=c=[M_{0}].$$
Since
$$b^{2n}=1,\quad c^{2}=bc+c,\quad bc=cb,$$
we have
$$\Phi (x_1^{2n}-1)=0,\quad\Phi(x_2^2-x_1x_2-x_2)=0, \quad\Phi(x_1x_2-x_2x_1)=0.$$
It follows that $\Phi (I)=0,$ and $\Phi$ induces a ring epimorphism
$$\overline{\Phi}: \mathbb{Z}\lr{x_1,x_2}/I\rightarrow r(H_{4n}),$$
such that $\overline{\Phi}(\overline{v})=\Phi (v)$ for all $v\in \mathbb{Z}\lr{x_1,x_2}$, where $\overline{v}=\pi(v)$ (natural epimorphism $\pi: \mathbb{Z}\lr{x_1,x_2}\rightarrow \mathbb{Z}\lr{x_1,x_2}/I$).
As $r(H_{4n})$ is a free $\mathbb{Z}$-module of rank $4n$, with a $\mathbb{Z}$-basis $\{b^{i} \mid 0\leq i\leq 2n-1\}\cup \{b^{j}c\mid 0 \leq j \leq 2n-1\}$, we can define a $\mathbb{Z}$-module homomorphism:
$$\Psi: r(H_{4n})\rightarrow \mathbb{Z}\lr{x_1,x_2}/I,$$
$$b^{i}c\rightarrow \overline{x_1^{i}x_2}=\overline{x_1}^{i}\overline{x_2},~~ b^{j}\rightarrow \overline{x_1^{j}}=\overline{x_1}^{j},~ 1 \leq i,j\leq 2n-1.$$
Observe that as a free $\mathbb{Z}$-module, $\mathbb{Z}\lr{x_1,x_2}/I$ is generated by elements $\overline{x_1^{i}x_2}$ and $\overline{x_1^{j}}, 0 \leq i,j\leq 2n-1, $ we have
$$\Psi\overline{\Phi}(\overline{x_1^{i}x_2})=\Psi\Phi(x_1^{i}x_2)=\Psi(b^{i}c)=
\overline{x_1^{i}x_2},$$
$$\Psi\overline{\Phi}(\overline{x_1^{j}})=\Psi\Phi(x_1^{j})=\Psi(b^{j})=
\overline{x_1^{j}},$$
for all $0 \leq i,j\leq 2n-1.$ Hence $\Psi\overline{\Phi}=id$, and $\overline{\Phi}$ is injective. Thus, $\overline{\Phi}$ is a ring isomorphism.
\end{proof}
\section{Weak Hopf algebras corresponding to $H_{4n}$}
 Firstly, we recall the concept of weak Hopf algebra given by Li(see \cite{LF}). By definition, a weak Hopf algebra is $k$-bialgebra $H$ with a map $T\in \hom(H, H)$ such that
$T\ast id \ast T =T$ and $id\ast T\ast id  =id $, where $\ast$ is the convolution map in $\hom(H, H)$.

Let $\w H_{4n}$ be the algebra generated by
$Z, X$ with relations
$$Z^{2n+1}=Z, \quad ZX=qXZ,\quad  X^2=0.$$
\begin{thm}\label{thm4-1}
$\w H_{4n}$ is a noncommutative and noncocommutative weak Hopf algebra with comultiplication, counit and the weak antipode $T$ as follows
\begin{eqnarray*}
  &&\Delta(Z)=Z\otm Z+a(1-q^{-2})Z^{n+1}X\otm ZX, \quad \Delta(X)=X\otm 1+Z^n\otm X; \\
  &&\epsilon(Z)=1, \quad \epsilon(X)=0,\\
 &&T(Z)=Z^{2n-1}, \quad  T(X)=-Z^nX.
\end{eqnarray*}
\end{thm}
\begin{proof}
Firstly, it can be shown by direct calculations that the following relations hold:
$$\Delta(Z)^{2n+1}=\Delta(Z), \quad \Delta(Z)\Delta(X)=q\Delta(X)\Delta(Z),\quad  \Delta(X)^2=0,$$
$$\epsilon(Z)^{2n+1}=\epsilon(Z), \quad \epsilon(Z)\epsilon(X)=q\epsilon(X)\epsilon(Z),\quad  \epsilon(X)^2=0,$$
Therefore, $\Delta$ and $\epsilon$ can be extended to algebra morphism from $\w H_{4n}$ to $\w H_{4n}\otimes \w H_{4n}$ and from $\w H_{4n}$ to $k$ respectively. We also have
$$(\Delta\otimes id)\Delta (Y)=(id\otimes\Delta)\Delta (Y),$$
$$(\epsilon\otimes id)\epsilon(Y)=(id\otimes\epsilon)\epsilon(Y)=Y$$
for $Y=X,Z$. It follows that $\w H_{4n}$ is a bialgebra.

Secondly, we prove that
in the bialgebra $\w H_{4n}$, the map $T$ can define a weak antipode in the natural way.
To see this, note that the map
$T: \w H_{4n}\to {\w H_{4n}}^{\rm op}$ keeps the defining relations:
$$(T(Z))^{2n+1}=((Z)^{2n-1})^{2n+1}=Z^{2n-1}=T(Z),$$
$$(T(X))^{2}=(-Z^{n}X)^{2}=0.$$
$$T(X)T(Z)=(-Z^{n}X)(Z)^{2n-1}=q^{1-2n}(Z)^{2n-1}(-Z^{n}x)=qT(Z)T(X).$$
It follows that the map $T$ can be extended to an anti-algebra homomorphism $T: \w H_{4n}\to \w H_{4n}$.
Besides, it is easy to see that in $\w H_{4n}$,
\begin{align*}
&(id\ast T\ast id)(Z)=ZT(Z)Z=Z^{2n+1}=Z=id(Z),\\
&(T\ast id\ast T)(Z)=T(Z)ZT(Z)=Z^{zn-1}=T(Z).
\end{align*}
and
\begin{align*}
(id\ast T\ast id)(X)=&\mu(id\otimes T\otimes id)(X\otimes 1\otimes 1+Z^n\otimes
X\otimes 1+Z^n\otimes Z^n\otimes X)\\=&X+z^{n}T(X)+z^{n}T(z)^nX=X-z^{2n}X+z^{2n}X=id(X),\\
(T\ast id\ast T)(X)=&\mu(T\otimes id\otimes T)(X\otimes 1\otimes 1+Z^n\otimes
X\otimes 1+Z^n\otimes Z^n\otimes X)\\=&T(X)+T(Z^n)X+T(Z^n)Z^nT(X)\\
=&-Z^{n} X+Z^{n} X -Z^{3n} X=-Z^{n} X=T(X).
\end{align*}
On the other hand, we have
   $$id\ast T(X)=X+Z^{n} T(X)=X -Z^{2n}X=X(1-Z^{2n}),$$
    $$T\ast id(X)=T(X)+T(Z)^n X=-Z^{n} X+ Z^{n}X=0.$$
   and
   $id\ast T(Z)=ZT(Z)+a(1-q^{-2})Z^{n+1}XT(ZX)=Z^{2n}+a(1-q^{-2})Z^{n+1}X(-Z^{n} X)Z^{2n-1} =Z^{2n}=T(Z).$

 These arguments show that for any $h\in \w H_{4n}$ we have $id\ast T(h)$ and $T\ast id(h)$ are in the center of $\w H_{4n}$.  Now, if $a, b\in \w H_{4n}$ and
\begin{eqnarray*}
&&T\ast id\ast T(a)=T(a), \  T\ast id\ast T(b)=T(b),\\
&&id\ast T\ast id(a)=a, \  id\ast T\ast id(b)=b,
\end{eqnarray*}
one can check that
 $$T\ast id\ast T(ab)=T(ab), \quad  id\ast T\ast id(ab)=ab.$$
 Hence $T$ is indeed define a weak antipode of $\w H_{4n}$ and $\w H_{4n}$ is a weak Hopf algebra, which is non-commutative and non-cocommutative.
\end{proof}

Let $J= Z^{2n}$, it is easy to see that $J$ and $1-J$ are a pair of orthogonal central
idempotents in $\w H_{4n}$. Let
$\w_1=\w H_{4n}J$, $\w_2=\w H_{4n}(1-J)$.
\begin{prop}\label{prop4-2} We have
$\w H_{4n}= \w_1\oplus\w_2$ as two-sided ideals. Moreover,
$\w_1\cong H_{4n}$ as Hopf algebras and $\w_2\cong k[y]/(y^2)$ as algebras.
\end{prop}
\begin{proof}
The first statement is easy to see. Let us prove the second one.

Note that $\w_1$ is generated by $Z$, $XJ$ and with $J$ as the identity and the relations
 $$JZ=ZJ=Z,\quad (XJ)^{2}=0,\quad Z(XJ)=q(XJ)Z.$$
Let $\rho: H_{4n}\rightarrow \w_1$ be the map defined by
$$\rho(1)=J,\quad \rho(z)=Z,\quad \rho(z^{-1})=Z^{2n-1}\quad \rho(x)=XJ.$$
It is straightforward to see that $\rho$ is well defined surjective algebraic homomorphism.
Let $\phi: \w H_{4n}\to H_{4n}$ be the map given by
$$\phi(1)=1, \quad\phi(X)=x, \quad\phi(Z)=z.$$
It is obvious that $\phi$ is a well defined algebra homomorphism. If we consider the restricted homomorphism ${\phi}|_{\w_1},$ then we have
${\phi}|_{\w_1}\circ \rho=id_{H_{4n}}$. Hence, $\rho$ is injective and
$\w_1\cong H_{4n}$ as algebras. Furthermore,
$\w_1$ is a Hopf algebra with comultiplication, counit and the antipode $S$ as follows
\begin{eqnarray*}
  &&\Delta(Z)=Z\otm Z+a(1-q^{-2})Z^{n+1}XJ\otm ZXJ, \quad \Delta(XJ)=XJ\otm 1+Z^n\otm XJ; \\
  &&\epsilon(Z)=1, \quad \epsilon(XJ)=0,\\
 &&S(Z)=Z^{2n-1}, \quad  S(XJ)=-Z^n XJ.
\end{eqnarray*}
It is clear that $\rho$ is a Hopf algebra isomorphism.
Now we prove that $\w_2\cong k[y]/(y^2)$. We first claim that $X(1-J)\ne 0$. Let $N$ be the $\w H_{4n}$-module with the basis
$\{ w_{1},w_2\}$. The action of $\w H_{4n}$ on $N$ is
given by
\begin{eqnarray*}
  &&Z\cdot w_{i}=0,\quad  i=1, 2. \\
  &&X\cdot w_{i}=\left\{
\begin{array}{ll}
w_{2}, \quad i=1,   \\

0,  \quad \quad i=2.
\end{array}
\right.
\end{eqnarray*}
It follows that $Jw_i=0$ for $i=1, 2$ and $[X(1-J)]w_1=w_2$. Therefore, we have
$X(1-J)\ne 0$ and $[X(1-J)]^2=0$.

Let $\phi_{1}: k[y]/(y^{2})\rightarrow \w_2$ be the map defined by
$$\phi_{1}(y)=X(1-J),\quad \phi_{1}(1)=1-J.$$
It is easy to show that $\phi_{1}$ is an algebraic isomorphism, and we
have $\w_2\cong k[y]/(y^{2}).$
\end{proof}

\section{Indecomposable representations of $\w H_{4n}$}

By Proposition \ref{prop4-2}, $\w H_{4n}= H_{4n}\oplus k[y]/(y^{2})$. Hence the indecomposable modules of $H_{4n}$ and $k[y]/(y^{2})$ constitute all the indecomposable $\w H_{4n}$-modules up to isomorphism.

For any $i\in\Z_{2n}$, let $S_i$ be the $1$-dimensional cyclic $\w H_{4n}$-module with base $\{v_i\}$, with the action $X\cdot v_i=0, Z\cdot v_i=q^iv_i,$ and $M_i$ be the $2$-dimensional cyclic $\w H_{4n}$-module with bases $\{v_{1}^i, v_{2}^i\}$. The module structures are as follows:
 \begin{eqnarray*}
X(v_1^{i},v_2^{i})&=&(v_1^{i},v_2^{i})\left(
                                          \begin{array}{cc}
                                            0 & 0 \\
                                            1 & 0 \\
                                          \end{array}
                                        \right),\\
Z(v_1^{i},v_2^{i})&=&(v_1^{i},v_2^{i})\left(
                                          \begin{array}{cc}
                                            q^{i} &0 \\
                                            0 & q^{i+1} \\
                                          \end{array}
                                        \right).
\end{eqnarray*}
 In fact,
$S_i$ and $M_i$ are just indecomposable $\w H_{4n}$-modules corresponding to those of  $H_{4n}$-modules.

Let $N_0$ be the k-vector space with a basis $w_0$, the actions of $\w H_{4n}$ on $N_{0}$ are
defined by $Z\cdot w_0=0, \quad X\cdot w_0=0$. Let $N_1$ be the $2$-dimensional $\w H_{4n}$-module with bases $\{w_{1}, w_{2}\}$. The module structures are as follows:
 \begin{eqnarray*}
X(w_{1}, w_{2})&=&(w_{1}, w_{2})\left(
                                          \begin{array}{cc}
                                            0 & 0 \\
                                            1 & 0 \\
                                          \end{array}
                                        \right),\\
Z(w_{1}, w_{2})&=&(w_{1}, w_{2})\left(
                                          \begin{array}{cc}
                                           0 &0 \\
                                            0 & 0 \\
                                          \end{array}
                                        \right).
\end{eqnarray*}
It is noted that $N_0$ and $N_1$ are just indecomposable $\w H_{4n}$-modules corresponding to those of  $k[y]/(y^{2})$-modules. Therefore, we have
\begin{prop} \label{prop5-1}
The set
$$\left\{S_i, M_i\mid  i\in \mathbb{Z}_{2n}\} \cup \{N_{j}
\mid j=0,1\right\}$$
 forms a complete list of non-isomorphic indecomposable
$\w H_{4n}$-modules.
\end{prop}
Now we establish the decomposition formulas of the tensor product of two indecomposable $\w H_{4n}$-modules.
 \begin{thm}\label{thm5-1} Let $i,j\in \Z_{2n}$, then as $\w H_{4n}$-modules, we have
\begin{enumerate}
\item
  $S_i \otimes S_j\cong S_{i+j(\mathrm{mod}2n)}\cong S_j \otimes S_i.$
\item
  $S_i \otimes M_j\cong M_{i+j(\mathrm{mod}2n)}\cong M_j \otimes S_i.$
\item
    $M_i \otimes M_j\cong M_{i+j(\mathrm{mod}2n)}\oplus M_{i+j+1(\mathrm{mod}2n)}\cong M_j \otimes M_i.$
\item
  $N_0 \otimes N_0\cong N_0\cong N_0 \otimes S_i\cong S_i\otimes N_0.$
\item
  $N_0 \otimes N_1\cong N_0\oplus N_0\cong N_0 \otimes M_i.$
\item
  $N_1 \otimes N_0\cong N_1\cong M_i \otimes N_0\cong N_1 \otimes S_i\cong S_i \otimes N_1.$
\item
  $N_1 \otimes N_1\cong N_1\oplus N_1\cong N_1 \otimes M_i\cong M_i \otimes N_1.$
\end{enumerate}
\end{thm}
\begin{proof}
Recall that $\Delta(X)=X\otm 1+Z^n\otm X$, $\Delta(Z)=Z\otm Z+a(1-q^{-2})Z^{n+1}X\otm ZX$. For $i\in \Z_{2n}$, let $v_i$ be the basis of $S_i$, $\{v_{1}^i, v_{2}^i\}$ be the basis of $M_i$, $\{w_0\}$ be the basis of $N_0$ and $\{w_{1}, w_{2}\}$ be the basis of $N_1$. Note that (1)-(3) can be obtained as \ref{thm3-1}.

(4). It is clear since for $i\in\Z_{2n}$, we have $X\cdot w_0\otm w_0=0=X\cdot w_0\otm v_i=X\cdot v_i\otm w_0$ and $X\cdot w_0\otm w_0=0=X\cdot w_0\otm v_i=X\cdot v_i\otm w_0$.

(5). Note that for $j,k\in\{1,2\}$ and $i\in\Z_{2n}$, $X\cdot w_0\otm w_j=0=X\cdot w_0\otm v_k^i$, and
$Z\cdot w_0\otm w_j=0=X\cdot w_0\otm v_k^i$,
so we have $N_0 \otimes N_1\cong N_0\oplus N_0\cong N_0 \otimes M_i.$

(6). Since for $j\in\{1,2\}$ and $i\in\Z_{2n}$,
\begin{eqnarray*}
&&X\cdot w_1\otm w_0=w_2\otm w_0, \quad X\cdot w_2\otm w_0=0, \quad Z\cdot w_j\otm w_0=0;\\
&&X\cdot v_1^i\otm w_0=v_2^i\otm w_0, \quad X\cdot v_2^i\otm w_0=0, \quad Z\cdot v_j^i\otm w_0=0;\\
&&X\cdot w_1\otm v_i=w_2\otm v_i,\quad X\cdot w_2\otm v_i=0, \quad Z\cdot w_j\otm v_i=0;\\
&&X\cdot v_i\otm w_j=(-1)^iv_i\otm w_2,\quad X\cdot v_i\otm w_2=0, \quad Z\cdot v_i\otm w_j =0
\end{eqnarray*}
it follows that $N_1 \otimes N_0\cong N_1\cong M_i \otimes N_0\cong N_1 \otimes S_i\cong S_i \otimes N_1.$

(7). For $j,k\in\{1,2\}$ and $i\in\Z_{2n}$, let $\sigma(i)=(-1)^i$.
\begin{eqnarray*}
&& X\cdot v_1^i\otimes w_{j}=\left\{
                         \begin{array}{ll}
                          v_2^i\otimes w_{1}+\sigma(i)v_1^i\otimes w_{2}, & j=1; \\
                         v_2^i\otimes w_{2}, & j=2.
                         \end{array}
                       \right.\\
&& X\cdot v_2^i\otimes w_{j}=\left\{
                         \begin{array}{ll}
                          \sigma(i+1)v_2^i\otimes w_{2}, & j=1; \\
                         0, & j=2.
                         \end{array}
                       \right.\\
&& Z\cdot v_k^i\otimes w_{j}=0.
\end{eqnarray*}
Therefore, if we set $\varpi_1=v_1^i\otimes w_{1}$, $\varpi_2=v_2^i\otimes w_{1}+\sigma(i)v_1^i\otimes w_{2}$, $\varpi_3=v_2^i\otimes w_{1}$, $\varpi_4=\sigma(i+1)v_2^i\otimes w_{2}$. Then we have
\begin{eqnarray*}
X\cdot\varpi_1=\varpi_2, \quad X\cdot\varpi_2=0, \quad X\cdot\varpi_3=\varpi_4, \quad X\cdot\varpi_4=0,\quad Z\cdot\varpi_l=0(l=1,2,3,4),
\end{eqnarray*}

and we obtain $ M_i \otimes N_1\cong N_1 \oplus N_1.$ Besides, note that
\begin{eqnarray*}
&& X\cdot w_{j}\otimes v_1^i=\left\{
                         \begin{array}{ll}
                          w_{2}\otimes v_1^i, & j=1; \\
                         0, & j=2.
                         \end{array}
                       \right.\\
&& X\cdot w_{j}\otimes v_2^i=\left\{
                         \begin{array}{ll}
                          w_{2}\otimes v_2^i, & j=1; \\
                         0, & j=2.
                         \end{array}
                       \right.\\
&& Z\cdot v_k^i\otimes w_{j}=0.
\end{eqnarray*}
wo we have $ N_1 \otimes M_i\cong N_1\oplus N_1$. Furthermore, take $w_k', k=1,2$ as another basis of $N_1$, then
\begin{eqnarray*}
&& X\cdot w_{j}\otimes w_k'=\left\{
                         \begin{array}{ll}
                          w_{2}\otimes w_k', & j=1; \\
                         0, & j=2.
                         \end{array}
                       \right.\\
&& Z\cdot w_{j}\otimes w_k'=0.
\end{eqnarray*}
wo we have $ N_1 \otimes N_1\cong N_1\oplus N_1$.
\end{proof}
Without confusion, we denote $[S_1]=b$, $[M_0]=c$, and $[N_0]=d$
\begin{cor} \label{cor5-3}
The Green ring $r(\w H_{4n})$ is a ring generated by $b$, $c$ and $d$. The set $\{b^{i}c^{j}\mid 0 \leq i \leq 2n-1,j=0,1\}\cup \{
c^{k}d \mid k=0,1\}$ forms a $\mathbb{Z}$-basis for $r(\w H_{4n})$.
\end{cor}
\begin{proof}
 By Theorem \ref{thm5-1}, $b^{2n}=1$ and $\{b^{i}=[S_i]\mid 0 \leq i \leq 2n-1\}$. Besides, for all $0 \leq i \leq 2n-1$, $[S_{i}]c=[M_{i}]$, hence $[M_{i}]=b^ic$ and all the two-dimensional simple $H_{4n}$ module $\{M_{i}\mid 0 \leq i \leq 2n-1\}$ are obtained. Note that $[N_{0}]=d$ and $N_{1}\cong M_0 \otimes N_0$, we have $[N_{1}]=cd$. The result is obtained.
\end{proof}
\begin{thm}\label{thm5-2} The Green ring $r(\w H_{4n})$ is isomorphic to the quotient ring of the ring $\mathbb{Z}\langle x_1,x_2,x_3\rangle $ module the ideal $I$ generated by the following elements
$$x_1^{2n}-1,\quad x_2^2-x_1x_2-x_2, \quad x_1x_2-x_2x_1,$$
$$x_3^{2}-x_3,\quad x_1x_3-x_3, \quad x_3x_1-x_3, \quad x_3x_2-2x_3.$$
\end{thm}
\begin{proof}
By Corollary \ref{cor3-2}, $r(\w H_{4n})$ is generated by $b $, $c$ and $d$. Hence there is a unique ring epimorphism
$$\Phi: \mathbb{Z}\langle x_1,x_2,x_3\rangle \rightarrow r(\w H_{4n})$$
such that
$$\Phi(x_1)=b=[S_{1}],\quad \Phi(x_2)=c=[M_{0}], \quad \Phi(x_3)=d=[N_{0}].$$
By Theorem \ref{thm5-1}
$$b^{2n}=1,\quad c^{2}=bc+c,\quad bc=cb,$$
$$d^{2}=d,\quad ad=da=d, \quad dc=2d.$$
Thus we have
$$\Phi (x_1^{2n}-1)=0,\quad\Phi(x_2^2-x_1x_2-x_2)=0, \quad\Phi(x_1x_2-x_2x_1)=0,$$
$$\Phi (x_3^{2}-x_3)=0,\quad\Phi(x_1x_3-x_3)=0, \quad\Phi(x_3x_1-x_3)=0,\quad \Phi(x_3x_2-2x_3).$$

It follows that $\Phi (I)=0,$ and $\Phi$ induces a ring epimorphism
$$\overline{\Phi}: \mathbb{Z}\langle x_1,x_2,x_3\rangle /I\rightarrow r(\w H_{4n}).$$
Comparing the rank of $\mathbb{Z}\langle x_1,x_2,x_3\rangle /I$ and $r(\w H_{4n}),$ it is easy to see that $\overline{\Phi}$ is a ring isomorphism.
\end{proof}
\section{The dual $H_{4n}^*$ of $H_{4n}$ and $\w H_{4n}^*$}
In this section, we consider the dual Hopf algebra $ H_{4n}^*$ of $H_{4n}$  and its weak Hopf algebra $\w H_{4n}^*$, we also describe the representation ring $r(\w H_{4n}^*)$ of $\w H_{4n}^*$.

Let $\alpha$ and $\eta$ be the linear forms on $H_{4n}$ defined on the basis $\{z^ix^j\}_{0\leq i<2n, j=0,1}$ by
\begin{center}$
\langle \alpha, z^ix^j \rangle=\delta_{j,0}q^i$ and $\langle \eta, z^ix^j \rangle=\delta_{j,1}q^i.$
\end{center}
 It is easy to determine that $H_{4n}^*$ is generated by $\alpha$ and $\eta$ with the following relations
\begin{eqnarray*}
    &&\alpha^{2n}=1, \quad \eta^2=a(1-\alpha^2), \quad \alpha \eta=-\eta \alpha,\\
  &&\Delta(\alpha)=\alpha\otm \alpha, \quad \Delta(\eta)=\eta\otm 1+\alpha\otm \eta; \\
  &&\epsilon(\alpha)=1, \quad \epsilon(\eta)=0,\\
 &&S(\alpha)=\alpha^{-1}, \quad  S(\eta)=-\alpha^{-1}\eta.
\end{eqnarray*}
Without lost of generality, we take $a=1$ and we get $\eta^2=1-\alpha^2$. The representations of $H_{4n}^*$ and their tensor products decompositions have been described in\cite{YANG2}, and the corresponding representation ring are obtained in \cite{WLZ}. By Theorem 8.2(\cite{WLZ}), the Green ring of $H_{4n}^*$ is a commutative ring generated by $Y, Z, X_1,\cdots, X_{n-1}$ with the relations
\begin{eqnarray*}
    &&Y^{2}=1, \quad Z^2=Z+YZ, \quad YX_1=X_1,\quad ZX_1=2X_1,\\
  && X_1^{j}=2^{j-1}X_j \quad  for \quad  1\leq j\leq n-1, \quad X_1^{n}=2^{n-2}Z^2
\end{eqnarray*}
Let $\w H_{4n}^*$ be the algebra generated by
$G, X$ with relations
$$G^{2n+1}=G, \quad GX=-XG,\quad  X^2=1-G^2.$$
Then
$\w H_{4n}^*$ is a noncommutative and noncocommutative bialgebra with comultiplication, counit as follows
\begin{eqnarray*}
  &&\Delta(G)=G\otm G, \quad \Delta(X)=X\otm 1+G\otm X; \\
  &&\epsilon(G)=1, \quad \epsilon(X)=0.
\end{eqnarray*}
Let $J= G^{2n}$, it is easy to see that $J$ and $1-J$ are a pair of orthogonal central
idempotents in $\w H_{4n}^*$. Let
$\w_1=\w H_{4n}^*J$, $\w_2=\w H_{4n}^*(1-J)$.
\begin{prop}\label{prop6-1} We have
$\w H_{4n}^*= \w_1\oplus\w_2$ as two-sided ideals. Moreover,
$\w_1\cong H_{4n}^*$ as Hopf algebras and $\w_2\cong k[y]/(y^2-1)$ as algebras.
\end{prop}
The proof is similar to Proposition \ref{prop4-2}, and we omit it here.
By Proposition \ref{prop6-1}, $\w H_{4n}^*= H_{4n}^*\oplus k[y]/(y^{2}-1)$. Hence up to isomorphism, the indecomposable $H_{4n}^*$ -modules and $k[y]/(y^{2}-1)$-modules constitute all the indecomposable $\w H_{4n}^*$-modules.

For $s=0$ or $n$, let $M[1,s]$ be the $1$-dimensional cyclic $\w H_{4n}^*$-module with the base $\{v_s\}$ defined by
 $X\cdot v_s=0,$ $G\cdot v_s=(-1)^{\frac{s}{n}}v_s.$  Let $M[2,s]$ be the $2$-dimensional cyclic $\w H_{4n}^*$-module with bases $\{v_{1}^s, v_{2}^s\}$ defined as follows
 \begin{eqnarray*}
X(v_1^{s},v_2^{s})&=&(v_1^{s},v_2^{s})\left(
                                          \begin{array}{cc}
                                            0 & 0 \\
                                            1 & 0 \\
                                          \end{array}
                                        \right),\\
G(v_1^{s},v_2^{s})&=&(v_1^{s},v_2^{s})\left(
                                          \begin{array}{cc}
                                            (-1)^{\frac{s}{n}} &0 \\
                                            0 & (-1)^{\frac{s}{n}+1} \\
                                          \end{array}
                                        \right).
\end{eqnarray*}
For $1\leq j\leq n-1$, let $P_j$  be the $2$-dimensional $\w H_{4n}^*$-module with bases $\{p_{1}^j, p_{2}^j\}$ and module structures as follows:
\begin{eqnarray*}
X(p_1^{j},p_2^{j})&=&(p_1^{j},p_2^{j})\left(
                                          \begin{array}{cc}
                                            0 & 1-q^{2j} \\
                                            1 & 0 \\
                                          \end{array}
                                        \right),\\
G(p_1^{j},p_2^{j})&=&(p_1^{j},p_2^{j})\left(
                                          \begin{array}{cc}
                                            q^{j} &0 \\
                                            0 & -q^{j} \\
                                          \end{array}
                                        \right).
\end{eqnarray*}
 In fact,
$M[k,s], k=1,2; s=0,n$ and $P_j, 1\leq j\leq n-1$ are just indecomposable $\w H_{4n}^*$-modules corresponding to those of $H_{4n}^*$-modules.

Let $N_i (i=0,1)$ be the k-vector space with a basis $w_i$, the actions of $\w H_{4n}^*$ on $N_{i}$ are
defined by $X\cdot w_i=(-1)^iw_i, \quad G\cdot w_i=0$.
It is noted that $N_0, N_1$ are just indecomposable $\w H_{4n}^*$-modules corresponding to those of  $k[y]/(y^{2}-1)$-modules. Therefore, we have
\begin{prop} \label{prop6-2}
The set
$$\left\{M[k,s], P_j\mid  k=1,2; s=0,n; 1\leq j\leq n-1 \} \cup \{N_{i}
\mid i=0,1\right\}$$
 forms a complete list of non-isomorphic indecomposable
$\w H_{4n}^*$-modules.
\end{prop}
 Now we establish the decomposition formulas of the tensor product of two indecomposable $\w H_{4n}^*$-modules.
 \begin{thm}\label{thm6-1} As $\w H_{4n}^*$-modules, we have
\begin{enumerate}
\item
 For $1\leq i,j\leq n-1$, $P_i \otimes P_j\cong\left\{
                         \begin{array}{ll}
                           M[2,0]\oplus M[2,n], & n\mid i+j; \\
                           2P_{i+j}, & n\nmid i+j.
                         \end{array}
                       \right.
$
\item
   For $k\in\{1,2\}, s\in \{0,n\}, 1\leq j\leq n-1$, $M[k,s] \otimes P_j\cong kP_j\cong P_j \otimes M[k,s].$
\item
    For $k,l\in\{1,2\}, s,j\in \{0,n\}$,

  $M[k,s]\otimes M[l,j]\cong\left\{
                         \begin{array}{ll}
                           M[2,0]\oplus M[2,n], & k+l=4; \\
                           M[k+l-1,s+j(\mathrm{mod}2n)], & k+l<4.
                         \end{array}
                       \right.$
\item
  For $i,j\in\{0,1\}$, $N_i \otimes N_j\cong N_i.$
\item
   For $k\in\{1,2\}, s\in \{0,n\}, j\in\{0,1\}$, $M[k,s] \otimes N_j\cong \left\{
                         \begin{array}{ll}
                           N_{j+\frac{s}{n}}, & k=1; \\
                           N_0\oplus N_1, & k=2.
                         \end{array}
                       \right.$
\item
  For  $k\in\{1,2\}, s\in \{0,n\}, j\in\{0,1\}$, $N_j \otimes M[k,s]\cong kN_{j}.$
\item
  For $i\in\{0,1\}, 1\leq j\leq n-1$, $N_i \otimes P_j \cong2N_i$, $P_j\otimes N_i\cong N_0\oplus N_1.$


\end{enumerate}
\end{thm}
\begin{proof}
Recall that $\Delta(G)=G\otm G$ and $\Delta(X)=X\otm 1+G\otm X$.
(1)-(3) can be proved proved similarly as in \cite{YANG2,WLZ}.

(4). Note that $G\cdot w_i=0, \quad X\cdot w_i=(-1)^iw_i$, therefore $N_i \otimes N_j\cong N_i.$ for $i,j\in\{0,1\}$.

(5) and (6). Let $k\in\{1,2\}, s\in \{0,n\}, j\in\{0,1\}$ and $v_s$ be the basis of $M[1,s]$, then $X\cdot v_s=0$ and $G\cdot v_s=(-1)^{\frac{s}{n}}v_s$, so we have
\begin{eqnarray*}
&&G\cdot (v_s\otm w_j)=0,\quad X\cdot (v_s\otm w_j)=(-1)^{(\frac{s}{n}+j)}v_s\otm w_j,\\
&&G\cdot (w_j\otm v_s)=0,\quad X\cdot (w_j\otm v_s)=(-1)^{j}w_j\otm v_s,
\end{eqnarray*}
hence $M[1,s] \otimes N_j\cong N_{j+\frac{s}{n}}$ and $ N_j\otimes M[1,s]\cong N_{j}.$

Let $\{v_{1}^s, v_{2}^s\}$ be the basis of $M[2,s]$, then
\begin{eqnarray*}
&&X\cdot (v_{1}^s\otm w_j+v_{2}^s\otm w_j+(-1)^{(\frac{s}{n}+j)}v_{1}^s\otm w_j)=v_{2}^s\otm w_j+(-1)^{(\frac{s}{n}+j)}v_{1}^s\otm w_j+v_{1}^s\otm w_j,\\
&&G\cdot (v_{1}^s\otm w_j+v_{2}^s\otm w_j+(-1)^{(\frac{s}{n}+j)}v_{1}^s\otm w_j)=0,\\
&&X\cdot(v_{1}^s\otm w_j-v_{2}^s\otm w_j-(-1)^{(\frac{s}{n}+j)}v_{1}^s\otm w_j)=v_{2}^s\otm w_j+(-1)^{(\frac{s}{n}+j)}v_{1}^s\otm w_j-v_{1}^s\otm w_j,\\
&&G\cdot(v_{2}^s\otm w_j+(-1)^{(\frac{s}{n}+j)}v_{1}^s\otm w_j)=0,
\end{eqnarray*}
therefore $M[2,s] \otimes N_j\cong N_{0}\oplus N_{1}$. Besides,
\begin{eqnarray*}
X\cdot (w_j\otm v_{1}^s)=(-1)^j(w_j\otm v_{1}^s),\quad X\cdot (w_j\otm v_{2}^s)=(-1)^j(w_j\otm v_{2}^s),\quad G\cdot (w_j\otm v_{k}^s)=0,
\end{eqnarray*}
therefore $ N_j\otimes M[2,s]\cong 2N_j$.

(7). Note that
\begin{eqnarray*}
&&X\cdot (p_{1}^j\otm w_i)= p_{2}^j\otm w_i+(-1)^iq^{j}p_{1}^j\otm w_i,\\
&&X\cdot (p_{2}^j\otm w_i)=(1-q^{2j})p_{1}^j\otm w_i-(-1)^iq^{j}p_{2}^j\otm w_i,
\end{eqnarray*}
Let $\omega_1=(1+(-1)^i q^{j})p_{1}^j\otm w_i+p_{2}^j\otm w_i, \omega_2=(1-(-1)^i q^{j})p_{1}^j\otm w_i-p_{2}^j\otm w_i$, then we have
$X\cdot \omega_1=\omega_1$, $X\cdot \omega_2=-\omega_2$ and $G\cdot \omega_k=0,k=0,1$.  Therefore $P_j\otimes N_i\cong N_0\oplus N_1$. Besides, $X\cdot (w_i\otm p_{k}^j)=(-1)^i w_i\otm p_{k}^j$ and $G\cdot (w_i\otm p_{k}^j)=0$ for $k=1,2$, therefore $ N_i\otimes P_j\cong 2N_i.$
\end{proof}

Denote $M[1,n]=b$, $M[2,0]=c$, $P_j=a_j, j\in\{1,2,\cdots n-1\}$, and $N_0=d$, then we have
\begin{cor} \label{cor6-2}
The Green ring $r(\w H_{4n}^*)$ is a ring generated by $b$, $c$, $d$ and $a_j$. The set $\{a_j, b^{i}c^{k}\mid 1 \leq j \leq n-1,i, k=0,1\}\cup \{
b^{i}d, \mid i=0,1\}$ forms a $\mathbb{Z}$-basis for $r(\w H_{4n}^*)$.
\end{cor}
\begin{proof}
 By Theorem \ref{thm6-1}, $b^{2}=1, bc=cb=M[2,n]$ and $c^{2}=c+bc$. Therefore, the set $\{a_j, b^{i}c^{k}\mid 1 \leq j \leq n-1,i,k=0,1\}$ has a one to one correspondence with the modules $\{M[k,s], P_j \}$. Besides, note that $d^2=d$, and $[N_{1}]=bd$, the result is obtained.
\end{proof}
\begin{thm}\label{thm6-2} The Green ring $r(\w H_{4n}^*)$ is isomorphic to the quotient ring of the ring $\mathbb{Z}\langle Y, Z, X_j, W\rangle$ module the ideal $I$ generated by the following elements
\begin{eqnarray}
&& \quad Y^{2}-1, \quad Z^2-Z-YZ, \quad YX_1-X_1, \quad ZX_1-2X_1, \quad YZ-ZY, X_1 Y-YX_1,  X_1 Z-ZX_1; \\
  &&\quad X_1^{j}-2^{j-1}X_j (1\leq j\leq n-1), \quad X_1^{n}-2^{n-2}Z^2;\\
&&\quad W^2-W, \quad WY-W, \quad WZ-2W, \quad ZW-W-YW, \quad WX_1-2W, \quad X_1W-W-YW.
\end{eqnarray}
\end{thm}
\begin{proof}
By Corollary \ref{cor6-2}, $r(\w H_{4n})^*$ is generated by $b $, $c$, $d$ and $a_j$. Hence there is a unique ring epimorphism
$$\Phi: \mathbb{Z}\langle Y, Z, X_j, W\rangle \rightarrow r(\w H_{4n}^*)$$
such that
$$\Phi(Y)=b,\quad \Phi(Z)=c, \quad \Phi(X_j)=a_j(1\leq j\leq n-1), \quad \Phi(W)=d.$$
By Theorem \ref{thm6-1}, it is easy to see that $\Phi$ vanishes at the generators of the ideal $I$ given by (5.1)-(5.3).
It follows that $\Phi$ induces a ring epimorphism
$$\overline{\Phi}: \mathbb{Z}\langle Y, Z, X_j, W\rangle/I\rightarrow r(\w H_{4n}^*).$$
Comparing the rank of $\mathbb{Z}\langle Y, Z, X_j, W\rangle /I$ and $r(\w H_{4n}^*),$ it is easy to see that $\overline{\Phi}$ is a ring isomorphism.
\end{proof}

\section*{Declarations}

\noindent\textbf{Ethical Approval.} This declaration is not applicable.\\
\textbf{Competing interests.} The authors declare that they have no conflict of interest.\\
\textbf{Authors' contributions.} S. Yang developed the idea for the study, all the authors did the analyses, and J. Chen and S. Yang wrote the paper.\\
\textbf{Funding.} This work was Supported by the National Natural
Science Foundation of China (Grant Nos.11701019, 11671024, 11871301),
the Beijing Natural Science Foundation (Grant No.1162002), and the Science and Technology Project of Beijing Municipal Education Commission(Grant No. KM202110005012).\\
\textbf{Availability of data and materials.} The data that support the findings of this study are available on request from the corresponding author upon reasonable request.


\begin{thebibliography}{111}

\bibitem{AIS}\label{AIS} N. Aizawa, P.S. Isaac. Weak Hopf algebras corresponding to $U_q (sl_n)$. J. Math. Phys.. 2003, 44, 5250-5267.

\bibitem{AC13}\label{AC13} N. Andruskiewitsch, J. Cuadra. On the structure of (co-Frobenius) Hopf algebras. J. Noncomm. Geom. 2013, 7(1): 83-104.

\bibitem{AN}\label{AN} N. Andruskiewitsch, S. Natale. Counting arguments for Hopf algebras of low dimension. Tsukuba J. Math. 2001, 25(1): 187-201.

\bibitem{BG}\label{BG} M. Beattie, G. A. Garcia. Classifying Hopf algebras of a given dimen-
sion. Contemp. Math. 2013, 585: 125-152.

\bibitem{COZ}\label{COZ}
H. Chen, F. V. Oystaeyen and Y. Zhang. The Green rings of Taft algebras. Proc. Amer. Math. Soc.. 2014, 142(3): 765-775.

\bibitem{CW}\label{CW} Q. Chen, D. Wang. Constructing Quasitriangular Hopf algebras. Comm. Alg. 2015,43(4): 1698-1722

\bibitem{CNG}\label{CNG} Y. L. Cheng, S. H. Ng. On Hopf algebras of dimension $4p$. J. Algebra. 2011, 328: 399-419.

\bibitem{CIBILS}\label{CIBILS} C. Cibils. A quiver quantum groups. Commun. Math. Phys.. 1993, 459-477.

\bibitem{GG}\label{GG} G. A. Garc$\acute{i}$a, J. M. Jury Giraldi. On Hopf algebras over quantum subgroups. J. Pure Appl. Algebra. 2019, 233(2): 738-768.


\bibitem{NG09}\label{NG09} M. Hilgemann, S. H. Ng. Hopf algebras of dimension 2$p^2$. J. Lond.
Math. Soc.. 2009, 80: 295-310.

\bibitem{HX16}\label{HX16} N. Hu, R. Xiong. Some Hopf algebras of dimension 72 without the Chevalley property.
arXiv: 1612.04987.

\bibitem{KA}\label{KA} C. Kassel. Quantum groups. Springer-Verlag, New York. 1995.

\bibitem{LF}\label{LF} F. Li. Weak Hopf algebras and new solutions of Yang-Baxter equation. J. Algebra. 1998, 208: 72-100.

\bibitem{LH}\label{LH} Y. Li, N. Hu. The Green rings of the 2-rank Taft algebra and its two relatives twisted.
J. Algebra. 2014, 410: 1-35.


\bibitem{LZ}\label{LZ} L. Li, Y. Zhang. The Green rings of the Generalized Taft algebras. Con. Math.. 2013, 585: 275-288.


\bibitem{LW2}\label{LW2} D. Lu, D. Wang. Ore extensions of quasitriangular Hopf group coalgebras. J. Algebra Appl.
2014, 13(6), 1450016.

\bibitem{MA}\label{MA}
S. Majid. Foundations of quantum group theory. Cambridge Univ. Press, Cambridge, 1995.

\bibitem{MONT}\label{MONT} S. Montgomery. Hopf Algebras and their actions on rings. CBMS series in Math.. 1993, vol.82, Am. Math. Soc., Providence.

\bibitem{NAT}\label{NAT} S. Natale. Hopf algebras of dimension 12. Algebr. Represent. Theory. 2002, 5: 445-455.

\bibitem{NA}\label{NA} A. Nenciu. Quasitriangular pointed Hopf algebras
constructed by Ore extensions. Algebra Represent. Theory. 2004, 7: 159-172.

\bibitem{NG02}\label{NG02} S. H. Ng. Non-semisimple Hopf algebras of dimension $p^2$. J. Algebra. 2002, 255: 182-197.

\bibitem{NG05}\label{NG05} S. H. Ng. Hopf algebras of dimension 2$p$. Proc. Amer. Math. Soc. 2005, 133 (8): 2237-2242.


\bibitem{RAD}\label{RAD}
D. E. Radford. On the coradical of a finite-dimensional Hopf algebra. Proc. Amer. Math. soc.. 1975, 53: 9-15.

\bibitem{DS}\label{DS} D. Stefan. Hopf Algebras of Low Dimension. J. Algebra. 1999, 211: 343-361.

\bibitem{SY2}\label{SY2} D. Su, S. Yang. Representation ring of small quantum group $\bar{U}_q{(sl_2)}$. J. Math. Phys.. 2017, 58: 091704.

\bibitem{SY}\label{SY}
D. Su, S. Yang. Automorphism group of representation ring of the weak Hopf algebra $\widetilde{H_8}$. Czech Math J. 2018, 68(143), 1131-1148.

\bibitem{SY1}\label{SY1}
D. Su, S. Yang. Green rings of weak Hopf algebras based on generalized Taft algebras. Period Math Hung. 2018, 76: 229-242.

\bibitem{SW}\label{SW} M. E. Sweedler. Hopf Algebras . Benjamin, New York, 1969.

\bibitem{Wa}\label{Wa} M. Wakui. Various structures associated to the representation
categories of eight-dimensional nonsemisimple Hopf Algebras. Algebra Represent. Theory. 2004, 7: 491-515.

\bibitem{WL}\label{WL} Z. Wang, L. Li. Ore extensions of quasitriangular Hopf algebras. Acta Math.Sci.2009, 29A(6): 1572-1579.

\bibitem{WLZ}\label{WLZ} Z. Wang, L. Li, Y. Zhang. Green rings of pointed rank one Hopf algebras of non-nilpotent type. J. Algebra. 2016, 449: 108-137.

\bibitem{RX}\label{RX} R. Xiong. Finite-dimensional Hopf algebras over the smallest non-pointed
basic Hopf algebra. arXiv: 1801.06205v1.

\bibitem{RX17}\label{RX17} R. Xiong. On Hopf algebras over the unique 12-dimensional Hopf algebra
without the dual Chevalley property. Comm. Algebra. 2019, 47(4): 1516-1540.

\bibitem{YANG2}\label{YANG2} S. Yang. Representation of simple pointed Hopf algebras. J. Algebra Appl.. 2004, 3(1): 91-104.

\bibitem{YANG1}\label{YANG1} S. Yang. Weak Hopf algebras corresponding to Cartan matrices. J. Math. Phys.. 2005, 46(073502): 1-18.

\bibitem{Zhu}\label{Zhu} Y. Zhu. Hopf algebras of prime dimension. Intern. Math. Res. Notes. 1994, 1: 53-59.

\end{thebibliography}
\end{document}